\RequirePackage{fix-cm}
\documentclass[smallextended]{svjour3}       
\smartqed  
\usepackage{graphicx}
\usepackage{amsmath,amssymb}
\usepackage{tikz}
\usetikzlibrary{decorations.markings,arrows,automata,arrows,backgrounds,snakes}
\usepackage{graphicx}

    \newcommand{\ZZ}    {\Bbb{Z}}
    \newcommand{\RR}    {\Bbb{R}}

\begin{document}

\title{The realization problem for discrete Morse functions on trees
}


\author{Yuqing Liu         \and
        Nicholas A. Scoville 
}


\institute{Y. Liu \at
              Department of Math and CS \\
              Ursinus College\\
              601 E. Main Street\\
              Collegeville, PA 19426
              \email{yuliu@ursinus.edu}           
           \and
           N. Scoville \at
              Department of Math and CS \\
              Ursinus College\\
              601 E. Main Street\\
              Collegeville, PA 19426
              \email{nscoville@ursinus.edu}
}

\date{Received: date / Accepted: date}

\maketitle

\begin{abstract}
We introduce a new notion of equivalence of discrete Morse functions on graphs called persistence equivalence.  Two functions are considered persistence equivalent if and only if they induce the same persistence diagram.  We compare this notion of equivalence to other notions of equivalent discrete Morse functions.  We then compute an upper bound for the number of persistence equivalent discrete Morse functions on a fixed graph and show that this upper bound is sharp in the case where our graph is a tree.  This is a version of the ``realization problem" of the persistence map. We conclude with an example illustrating our construction.
\end{abstract}

\section{Introduction}

Since its inception in the early 2000s, persistent homology has almost single handedly brought topology to the forefront of mathematics.  A panacea of sorts, persistence has been used to study statistical mechanics \cite{mech17}, hypothesis testing \cite{Hyp14}, image analysis \cite{Carl09}, complex networks \cite{Sco17}, and many other phenomena. Part of the utility of persistent homology is that it acts as a complete characterization for certain collections of simplicial complexes, as there have been several recent results allowing one to reconstruct a simplicial complex from certain collections of persistence diagrams \cite{Turner2014}, \cite{Fasy2018}. Hence the collection of all persistence diagrams of a fixed simplicial complex is an important object to understand.\\

\noindent One means of studying this object is to study the persistence map, which is the map that takes functions on $X$ to their associated persistence diagram. In \cite{Curry-2017}, Justin Curry proposes what he calls the Realization Problem which asks ``given a space $X$ and particular type of function $f$, what is the image of the persistence map?" He answers this question in the case where $X$ is an interval and moreover, characterizes the fiber of this map.  In this paper, we study the Realization Problem in a kind of analogous case to the one studied by Curry; that is, we study it when $X$ is an abstract $1$-dimensional connected simplicial complex with $b_1(X)=0$ i.e., a tree, along with $f$ a discrete Morse function on the tree. A close cousin of persistent homology, discrete Morse theory is a topological tool due to Robin Forman \cite{Forman-95} \cite{F-02} that can be used to simplify a simplicial complex. Among other things, a discrete Morse function on a simplicial complex naturally gives rise to a filtration, and a filtration gives rise to a persistence diagram. Once we have a persistence diagram induced by the discrete Morse function, we introduce and study a new notion of equivalence of discrete Morse functions on graphs (although this may be defined on any simplicial complex) called \textbf{persistence equivalence}.  Two discrete Morse functions on a graph $G$ are called persistence equivalent if $D_f=D_g$ where $D_f$ is the persistence diagram induced by $f$.  We then count and construct all discrete Morse functions up to persistence equivalence on a fixed tree.\\

\noindent The idea is that it is easy to compute the persistent homology of a discrete Morse function on a tree without having to resort to a matrix.  This is done in Lemma \ref{lem compute persistence}.  The critical vertices birth new components while a critical edge connecting two trees kills a bar corresponding to the tree whose minimum value is greater than the minimum value of the other tree. We give the background in discrete Morse theory and persistent homology and introduce persistence equivalence in Section \ref{Graphs, discrete Morse theory, and persistence}.  In Section \ref{Relation with other notions of equivalence}, we compare our new notion of equivalence with ones currently found in the literature. Using a slight variation of the standard definition of a discrete Morse functions in order to ensure finiteness, we will count the total number of discrete Morse functions with a fixed number of critical values up to persistence equivalence on a tree. This is accomplished by providing a combinatorial upper bound in Corollary \ref{cor tree count}. Theorem \ref{thm counting on trees} then provides a method to construct any such discrete Morse function with a desired barcode.  We end in Section \ref{An example} with an example illustrating our construction.

\section{Graphs, discrete Morse theory, and persistence}\label{Graphs, discrete Morse theory, and persistence}

In this section, we introduce the background and notation that is needed throughout the body of this paper. We begin by reviewing the basics of graph theory.

\subsection{Graphs}

A \textbf{graph} $G=(V,E)$ is a non-empty finite set $V$ along with a symmetric, irreflexive relation $E$ on $V$. The set $V=V(G)$ is called the \textbf{vertex set} of $G$
while $E=E(G)$ is called the \textbf{edge set}. If $(u,v)\in E$, we write $e=uv$ or  $u,v< e$ to denote the edge $e$ with \textbf{endpoints} $u$ and $v$.  In this case, we say that $u$ and $v$ are \textbf{adjacent} while $e$ and $u$ are \textbf{incident}. We will use \textbf{simplex} (plural: \textbf{simplices}) to refer to a vertex or an edge of $G$, and use a Greek letter such as $\sigma$ to denote either a vertex or an edge. If a vertex $v$  is incident with exactly one edge $e$, the pair $\{v,e\}$ is called a \textbf{free pair} or a \textbf{leaf}.\\

\noindent A \textbf{path} in $G$ is a list $v_1, e_1, v_2, e_2, \ldots, v_k, e_k, v_{k+1}$ of vertices and edges such that the edge $e_i$ has endpoints $v_{i}$ and $v_{i+1}$ for $1\leq i \leq k$.  We further require that no edge is repeated. If there is a path between any two vertices of $G$, we define $G$ to be \textbf{connected}.\\

\noindent A \textbf{cycle} in a graph is a path with at least three edges that begins and ends at the same vertex and never repeats a vertex (other than the starting and ending vertex). In other words, a \textbf{cycle} is a path $v_1, e_1, v_2, \ldots, e_k, v_{k+1}$ such that no vertex is repeated other than $v_1=v_{k+1}$.\\

\noindent Because our main construction in Theorem \ref{thm counting on trees} is on a special kind of graph called a tree, we recall several important characterizations of trees.  They will be utilized without further reference.

\begin{theorem}\label{thm char of trees}(Characterization of trees) Let $G$ be a connected graph with $v$ vertices and $e$ edges. The following are equivalent:
\begin{enumerate}
    \item[a)] Every two vertices of $G$ are connected by a unique path.
    \item[b)] $v=e+1$
    \item[c)] $G$ contains no cycles
    \item[d)] $b_1(G)=0$
    \item[e)] The removal of any edge from $G$ results in a disconnected graph.
\end{enumerate}

\end{theorem}

\noindent A connected graph that satisfies any of the above characterizations is called a \textbf{tree}. Proofs of the equivalence of the statements may be found in any graph theory textbook(e.g. \cite[Chapter 2.2]{Chartrand16}).

\subsection{Discrete Morse theory on graphs}

We use a slightly more restrictive definition of a discrete Morse function than is normally given in the literature. For a discussion on the reason for the choices made in the definition, see Remark \ref{rem dmf}.

\begin{definition}\label{defn DMF} Let $G$ be a graph with $n$ vertices and edges,  $f\colon G\to [0,n]$ a function.  Then $f$ is \textbf{monotone} if whenever $v < e$, then $f(v) \leq f(e)$. We say $f$ is a \textbf{discrete Morse function} if $f$ is a monotone function with $\min(f)=0$ which is at most $2-1$ where if $f(\sigma) = f(\tau)$, then $\tau< \sigma$ or $\sigma<\tau$.  Furthermore, we require that if $f$ is 1-1 on $f(\sigma)$, then $f(\sigma)\in \mathbb{N}.$ Such a value is called a \textbf{critical value} and $\sigma$ is a \textbf{critical vertex (critical $0$-simplex)} or a \textbf{critical edge (critical $1$-simplex).} If $\sigma$ is not critical, $\sigma$ is called \textbf{regular}. 
\end{definition}

\noindent Under this definition, it is easy to see that $\min\{f\}=0$ will always be a critical value.  It is also not difficult to show \cite[Lemma 2.5]{Forman-95} that regular simplices come in pairs.  Further embedded in Definition \ref{defn DMF} is the fact that the vertex/edge regular pair is given the same value under the discrete Morse function.  This condition is called \textbf{flat}. For those concerned that our definition is too restrictive, Uli Bauer has shown that every discrete Morse function is homologically equivalent to one which is flat \cite[Proposition 2.19]{BauerThesis}.\\

\noindent One of the fundamental results in discrete Morse theory is the (weak) discrete Morse inequalities, relating the number of critical simplices of a discrete Morse function to the Betti numbers.  We will utilize this theorem in Lemma \ref{lem crit even}.

\begin{theorem}\label{thm weak dmi}\cite[Cor 3.7]{Forman-95}(Weak discrete Morse inequalities for graphs) Let $G$ be a graph and $f$ a discrete Morse function of $G$ with the number of critical $i$-simplices of $f$ denoted by $m_i$, $i=0,1$.  Then\\

\noindent (i) $m_0\geq b_0$ and $m_1\geq b_1$ where $b_i$ denotes the $i^{th}$ Betti number of $G$.\\
(ii) $b_0-b_1=m_0-m_1$.
\end{theorem}

\noindent Let $f\colon G \to [0,n]$ be a discrete Morse function. For any $a\in \RR$, we define \textbf{level subcomplex $a$} by $G_a:=\{\sigma\in G \colon f(\sigma)\leq a\}$.  Note that by the fact that $f$ is flat, $G_a$ is always a subgraph. If $0=c_0<c_1<\ldots < c_{m-1}$ are the critical values of $f$, we consider the sequence of subcomplexes
$$
\{v\}=G_{c_0}\subseteq G_{c_1}\subseteq \ldots \subseteq G_{c_{m-1}}
$$
called the \textbf{filtration of $G$ induced by the discrete Morse function $f$}. This induced filtration will be used in the next section.

\subsection{Persistent homology}

Let $G$ be a graph.  Suppose we have a filtration

$$G_0\subseteq \ldots \subseteq G_{m-1}.$$

\noindent For $i\leq j$, there is an inclusion function $f^{i,j}\colon G_i\to G_j$.  Passing to homology, we obtain a linear transformation $f_p^{i,j}\colon H_p(G_i)\to H_p(G_j)$.  The \textbf{$p^{th}$ persistent homology groups}, denoted $H^{i,j}_p$, is defined by $H^{i,j}_p:=\mathrm{im} (f^{i,j}_p)$.   The \textbf{$p^{th}$-persistent Betti numbers} are the corresponding Betti numbers, $\beta_p^{i,j}:=\mathrm{rank} H^{i,j}_p$.\\

 \noindent A class $[\alpha]\in H_p(G_i)$ is said to be \textbf{born at $G_i$ or at time $i$} if $[\alpha]$ is not in the image of $f_p^{i-1,i}$. A class $[\alpha]\in H_p(G_i)$ is said to \textbf{die at $G_{j+1}$ or at time $j+1$} if $f^{i,j}_p([\alpha])$ is not in the image of $f_p^{i-1,j}$ but $f_p^{i,j+1}([\alpha])$ is in the image of $f^{i-1,j+1}_p$. If $\alpha$ is born at $i$ and dies at $j+1$, we call $(i,j+1)$ a \textbf{persistence pair}. If $\sigma$ is born and never dies, then $\sigma$ is called a \textbf{point at infinity}. Plotting all persistence pairs in the Euclidean plane along with all points at infinity (represented by a $y$-value greater than the maximum death time) yields the \textbf{persistence diagram induced by the filtration}, denoted $D$.\\

\noindent As noted above, a discrete Morse function with critical values $c_0<c_1<\ldots c_{m-1}$ induces a filtration
$$
G_{c_0}\subseteq G_{c_1}\subseteq \ldots G_{c_{m-1}}
$$
where $G_{c_i}$ is the level subcomplex of $G$ at level $c_i$.  Hence, a discrete Morse function $f$ induces a persistence diagram $D_f$.\\

\noindent We are now ready to give the main object of study in this paper.

\begin{definition}Two discrete Morse functions $f,g\colon G \to [0,n]$ are \textbf{persistence equivalent} if $D_{f}= D_{g}$.
\end{definition}

\noindent It is standard in the literature to interchange the persistence diagram and corresponding barcode when there is no possibility of confusion. Hence, another viewpoint that we will adopt in this paper is to consider two functions persistence equivalent if their corresponding barcodes are equal, where equality is given up to permutation of the vertical stacking of the bars.

\begin{remark}\label{rem dmf}
When computing persistence, the definition of a discrete Morse function (Definition \ref{defn DMF}) ensures that all births and deaths (corresponding to critical values) occur only at integer values and that furthermore, the first birth occurs at time $0$ and that the barcode is completed by time $n$, where $n$ is the number of simplices of $G$.  In addition, there can be at most one event (either a birth or a death) at any time.  So for example, the following could be a barcode induced by a discrete Morse function on a graph with $11$ simplices

$$
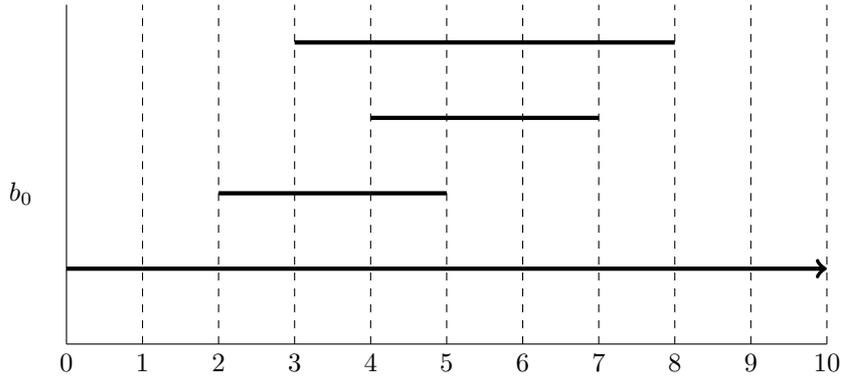
\begin{tikzpicture}[scale=.5][thick]
    \draw (1,1) -- (21,1);
    \draw (1,1) -- (1,10);
    \draw[thin, dashed] (3,1) -- (3,10);
    \draw[thin, dashed] (5,1) -- (5,10);
    \draw[thin, dashed] (7,1) -- (7,10);
    \draw[thin, dashed] (9,1) -- (9,10);
    \draw[thin, dashed] (11,1) -- (11,10);
    \draw[thin, dashed] (13,1) -- (13,10);
    \draw[thin, dashed] (15,1) -- (15,10);
    \draw[thin, dashed] (17,1) -- (17,10);
    \draw[thin, dashed] (19,1) -- (19,10);
    \draw[thin, dashed] (21,1) -- (21,10);

    \node at (1,.5) {$0$};
    \node at (3,.5) {$1$};
    \node at (5,.5) {$2$};
    \node at (7,.5) {$3$};
    \node at (9,.5) {$4$};
    \node at (11,.5) {$5$};
    \node at (13,.5) {$6$};
    \node at (15,.5) {$7$};
    \node at (17,.5) {$8$};
    \node at (19,.5) {$9$};
    \node at (21,.5) {$10$};

    \node[black] at (-.2,5) {$b_0$};

    \draw[black, ultra thick, ->] (1,3) -- (21,3);
    \draw[black, ultra thick] (5,5) -- (11,5);
    \draw[black, ultra thick] (9,7) -- (15,7);
    \draw[black, ultra thick] (7,9) -- (17,9);

\end{tikzpicture}
$$
but this one could not

$$
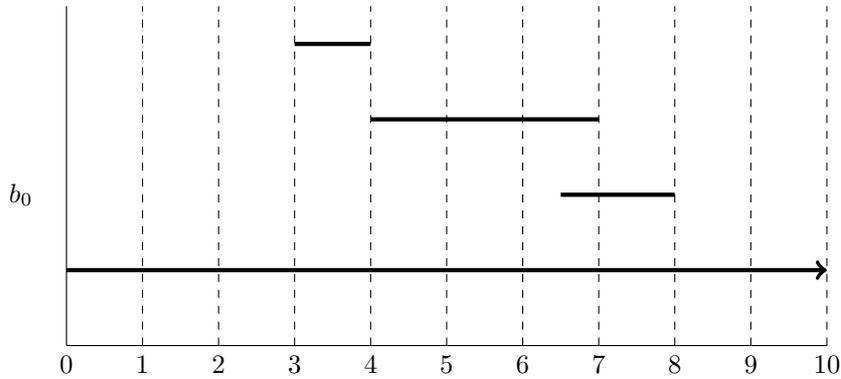
\begin{tikzpicture}[scale=.5][thick]
    \draw (1,1) -- (21,1);
    \draw (1,1) -- (1,10);
    \draw[thin, dashed] (3,1) -- (3,10);
    \draw[thin, dashed] (5,1) -- (5,10);
    \draw[thin, dashed] (7,1) -- (7,10);
    \draw[thin, dashed] (9,1) -- (9,10);
    \draw[thin, dashed] (11,1) -- (11,10);
    \draw[thin, dashed] (13,1) -- (13,10);
    \draw[thin, dashed] (15,1) -- (15,10);
    \draw[thin, dashed] (17,1) -- (17,10);
    \draw[thin, dashed] (19,1) -- (19,10);
    \draw[thin, dashed] (21,1) -- (21,10);

    \node at (1,.5) {$0$};
    \node at (3,.5) {$1$};
    \node at (5,.5) {$2$};
    \node at (7,.5) {$3$};
    \node at (9,.5) {$4$};
    \node at (11,.5) {$5$};
    \node at (13,.5) {$6$};
    \node at (15,.5) {$7$};
    \node at (17,.5) {$8$};
    \node at (19,.5) {$9$};
    \node at (21,.5) {$10$};

    \node[black] at (-.2,5) {$b_0$};

    \draw[black, ultra thick, ->] (1,3) -- (21,3);
    \draw[black, ultra thick] (14,5) -- (17,5);
    \draw[black, ultra thick] (9,7) -- (15,7);
    \draw[black, ultra thick] (7,9) -- (9,9);

\end{tikzpicture}
$$
since there is a birth and a death at $t=4$ and furthermore, there is a birth at $t=6.5\not \in \ZZ$.

\end{remark}

\noindent Hence, given these conditions, the total number barcodes that one can obtain from a discrete Morse function on a fixed graph is finite.  In Section \ref{Counting persistence equivalence classes}, we will give an upper bound for the number of persistence equivalent discrete Morse functions on a graph, and show that this estimate is sharp in the special case where $G=T$ is a tree.  Bur first, we compare persistence equivalence with other notions of equivalence.

\section{Relation with other notions of equivalence}\label{Relation with other notions of equivalence}

\subsection{Forman equivalence}

Recall that two discrete Morse functions $f,g\colon G \to [0,n]$ defined on a graph are \textbf{Forman equivalent} if and only if $V_f=V_g$, where $V_f$ is the induced gradient vector field of $f$ \cite{Forman-95}. It is easy to see that neither persistence nor Forman equivalence imply each other.

\begin{example}
The following two discrete Morse functions are persistence equivalent but not Forman equivalent.

$$
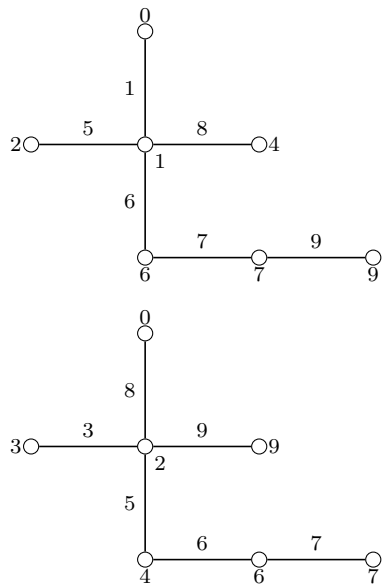
\begin{tikzpicture}[scale=.75]

\node[inner sep=2pt, circle] (7) at (2,2) [draw] {};
\node[inner sep=2pt, circle] (9) at (4,2) [draw] {};
\node[inner sep=2pt, circle] (6) at (0,2) [draw] {};
\node[inner sep=2pt, circle] (1) at (0,4) [draw] {};
\node[inner sep=2pt, circle] (0) at (0,6) [draw] {};
\node[inner sep=2pt, circle] (2) at (-2,4) [draw] {};
\node[inner sep=2pt, circle] (4) at (2,4) [draw] {};

\path[style=semithick] (0) edge node[anchor=east]{\small{$1$}}(1);
\path[style=semithick] (1) edge node[anchor=east]{\small{$6$}}(6);
\path[style=semithick] (2) edge node[anchor=south]{\small{$5$}}(1);
\path[style=semithick] (1) edge node[anchor= south]{\small{$8$}}(4);
\path[style=semithick] (6) edge node[anchor=south]{\small{$7$}}(7);
\path[style=semithick] (7) edge  node[anchor= south]{\small{$9$}}(9);

\node[anchor = south ]  at (0) {\small{$0$}};
\node[anchor = north west]  at (1) {\small{$1$}};
\node[anchor =  east]  at (2) {\small{$2$}};
\node[anchor = west]  at (4) {\small{$4$}};
\node[anchor = north]  at (6) {\small{$6$}};
\node[anchor = north]  at (7) {\small{$7$}};
\node[anchor = north]  at (9) {\small{$9$}};
\end{tikzpicture}
$$
$$
\begin{tikzpicture}[scale=.75]

\node[inner sep=2pt, circle] (7) at (2,2) [draw] {};
\node[inner sep=2pt, circle] (9) at (4,2) [draw] {};
\node[inner sep=2pt, circle] (6) at (0,2) [draw] {};
\node[inner sep=2pt, circle] (1) at (0,4) [draw] {};
\node[inner sep=2pt, circle] (0) at (0,6) [draw] {};
\node[inner sep=2pt, circle] (2) at (-2,4) [draw] {};
\node[inner sep=2pt, circle] (4) at (2,4) [draw] {};

\path[style=semithick] (0) edge node[anchor=east]{\small{$8$}}(1);
\path[style=semithick] (1) edge node[anchor=east]{\small{$5$}}(6);
\path[style=semithick] (2) edge node[anchor=south]{\small{$3$}}(1);
\path[style=semithick] (1) edge node[anchor= south]{\small{$9$}}(4);
\path[style=semithick] (6) edge node[anchor=south]{\small{$6$}}(7);
\path[style=semithick] (7) edge  node[anchor= south]{\small{$7$}}(9);

\node[anchor = south ]  at (0) {\small{$0$}};
\node[anchor = north west]  at (1) {\small{$2$}};
\node[anchor =  east]  at (2) {\small{$3$}};
\node[anchor = west]  at (4) {\small{$9$}};
\node[anchor = north]  at (6) {\small{$4$}};
\node[anchor = north]  at (7) {\small{$6$}};
\node[anchor = north]  at (9) {\small{$7$}};
\end{tikzpicture}
$$

\noindent Using the exact same graphs, the example below also shows that Forman equivalence does not imply persistence equivalence.
$$
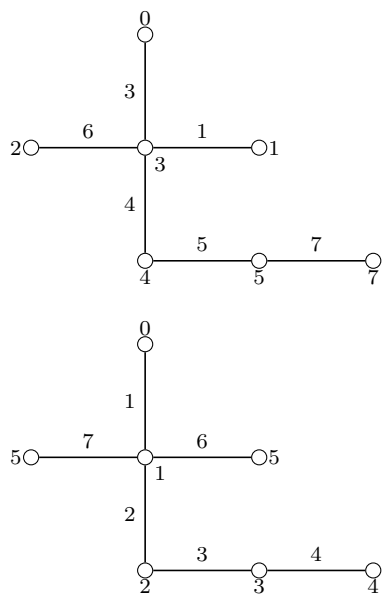
\begin{tikzpicture}[scale=.75]

\node[inner sep=2pt, circle] (7) at (2,2) [draw] {};
\node[inner sep=2pt, circle] (9) at (4,2) [draw] {};
\node[inner sep=2pt, circle] (6) at (0,2) [draw] {};
\node[inner sep=2pt, circle] (1) at (0,4) [draw] {};
\node[inner sep=2pt, circle] (0) at (0,6) [draw] {};
\node[inner sep=2pt, circle] (2) at (-2,4) [draw] {};
\node[inner sep=2pt, circle] (4) at (2,4) [draw] {};

\path[style=semithick] (0) edge node[anchor=east]{\small{$3$}}(1);
\path[style=semithick] (1) edge node[anchor=east]{\small{$4$}}(6);
\path[style=semithick] (2) edge node[anchor=south]{\small{$6$}}(1);
\path[style=semithick] (1) edge node[anchor= south]{\small{$1$}}(4);
\path[style=semithick] (6) edge node[anchor=south]{\small{$5$}}(7);
\path[style=semithick] (7) edge  node[anchor= south]{\small{$7$}}(9);

\node[anchor = south ]  at (0) {\small{$0$}};
\node[anchor = north west]  at (1) {\small{$3$}};
\node[anchor =  east]  at (2) {\small{$2$}};
\node[anchor = west]  at (4) {\small{$1$}};
\node[anchor = north]  at (6) {\small{$4$}};
\node[anchor = north]  at (7) {\small{$5$}};
\node[anchor = north]  at (9) {\small{$7$}};
\end{tikzpicture}
$$
$$
\begin{tikzpicture}[scale=.75]

\node[inner sep=2pt, circle] (7) at (2,2) [draw] {};
\node[inner sep=2pt, circle] (9) at (4,2) [draw] {};
\node[inner sep=2pt, circle] (6) at (0,2) [draw] {};
\node[inner sep=2pt, circle] (1) at (0,4) [draw] {};
\node[inner sep=2pt, circle] (0) at (0,6) [draw] {};
\node[inner sep=2pt, circle] (2) at (-2,4) [draw] {};
\node[inner sep=2pt, circle] (4) at (2,4) [draw] {};

\path[style=semithick] (0) edge node[anchor=east]{\small{$1$}}(1);
\path[style=semithick] (1) edge node[anchor=east]{\small{$2$}}(6);
\path[style=semithick] (2) edge node[anchor=south]{\small{$7$}}(1);
\path[style=semithick] (1) edge node[anchor= south]{\small{$6$}}(4);
\path[style=semithick] (6) edge node[anchor=south]{\small{$3$}}(7);
\path[style=semithick] (7) edge  node[anchor= south]{\small{$4$}}(9);

\node[anchor = south ]  at (0) {\small{$0$}};
\node[anchor = north west]  at (1) {\small{$1$}};
\node[anchor =  east]  at (2) {\small{$5$}};
\node[anchor = west]  at (4) {\small{$5$}};
\node[anchor = north]  at (6) {\small{$2$}};
\node[anchor = north]  at (7) {\small{$3$}};
\node[anchor = north]  at (9) {\small{$4$}};
\end{tikzpicture}
$$

\end{example}

\subsection{Homological equivalence}

In \cite{A-F-F-V-09}, Ayala et al. introduced the notion of homological equivalence and counted the number of discrete Morse functions up to homological equivalence in \cite{A-F-V-11} on all graphs.

\begin{definition}  Two discrete Morse functions $f$ and $g$ defined on a graph $G$ with critical values $a_0<a_1<\ldots < a_{m-1}$ and $c_0<c_1<\ldots < c_{m-1}$ respectively are \textbf{homologically equivalent} if $b_0(G_{a_i})=b_0(G_{c_i})$ and $b_1(G_{a_i})=b_1(G_{c_i})$ for all $0\leq i \leq m-1$.
\end{definition}

\noindent From the definitions, the following is immediate.

\begin{proposition}
If $f$ and $g$ are persistence equivalent, then $f$ and $g$ are homologically equivalent.
\end{proposition}

\noindent Of course, the converse is clearly false, as the following simple example illustrates.

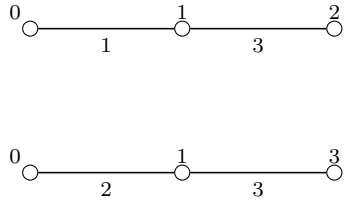
\begin{example}
$$
\begin{tikzpicture}

\node[inner sep=2pt, circle] (0) at (0,0) [draw] {};
\node[inner sep=2pt, circle] (1) at (2,0) [draw] {};
\node[inner sep=2pt, circle] (2) at (4,0) [draw] {};

\path[style=semithick] (0) edge node[anchor=north]{\small{$1$}}(1);
\path[style=semithick] (1) edge
node[anchor= north]{\small{$3$}}(2);

\node[anchor = south east]  at (0) {\small{$0$}};
\node[anchor = south]  at (1) {\small{$1$}};
\node[anchor = south]  at (2) {\small{$2$}};

\end{tikzpicture}
$$

$$
\begin{tikzpicture}

\node[inner sep=2pt, circle] (0) at (0,0) [draw] {};
\node[inner sep=2pt, circle] (1) at (2,0) [draw] {};
\node[inner sep=2pt, circle] (3) at (4,0) [draw] {};

\path[style=semithick] (0) edge node[anchor=north]{\small{$2$}}(1);
\path[style=semithick] (1) edge node[anchor= north]{\small{$3$}}(3);

\node[anchor = south east]  at (0) {\small{$0$}};
\node[anchor = south]  at (1) {\small{$1$}};
\node[anchor = south]  at (3) {\small{$3$}};

\end{tikzpicture}
$$
\end{example}

\subsection{Graph equivalence}\label{Graph equivalence}

The second author introduced the following notion of equivalence of discrete Morse functions on graphs in \cite{Aaronson2014}.

\begin{definition}\label{graph}
        Let $f,g \colon G \rightarrow [0,n]$ be two discrete Morse functions on a graph G with critical values $a_{0},a_{1},...,a_{m-1}$ and $c_{0},c_{1},...,c_{m-1}$ respectively.  The functions $f$ and $g$ are said to be \textbf{graph equivalent} if $G_{a_{i}}\cong G_{c_{i}}$ for every $0 \leq i \leq m-1$; that is, each level subcomplex is isomorphic as graphs.
\end{definition}

\noindent It should be noted that this is not to be confused with the notion of graph equivalence recently introduced by Curry \cite{Curry-2017}.\\

\noindent Although graph equivalence is quite stringent, two discrete Morse functions which are graph equivalent are not necessarily persistence equivalent.

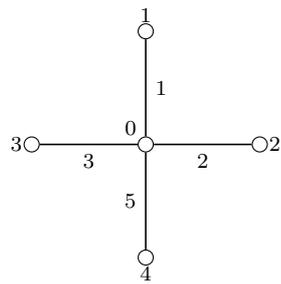
\begin{example}

$$
\begin{tikzpicture}[scale=.75]

\node[inner sep=2pt, circle] (4) at (0,0) [draw] {};
\node[inner sep=2pt, circle] (0) at (0,2) [draw] {};
\node[inner sep=2pt, circle] (1) at (0,4) [draw] {};
\node[inner sep=2pt, circle] (3) at (-2,2) [draw] {};
\node[inner sep=2pt, circle] (2) at (2,2) [draw] {};

\path[style=semithick] (1) edge node[anchor=west]{\small{$1$}}(0);
\path[style=semithick] (0) edge node[anchor=east]{\small{$5$}}(4);
\path[style=semithick] (3) edge node[anchor=north]{\small{$3$}}(0);
\path[style=semithick] (0) edge node[anchor= north]{\small{$2$}}(2);

\node[anchor = south east]  at (0) {\small{$0$}};
\node[anchor = south]  at (1) {\small{$1$}};
\node[anchor = east]  at (3) {\small{$3$}};
\node[anchor = west]  at (2) {\small{$2$}};
\node[anchor = north]  at (4) {\small{$4$}};

\end{tikzpicture}
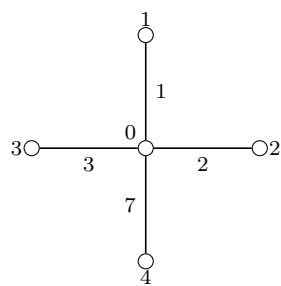
$$

$$
\begin{tikzpicture}[scale=.75]

\node[inner sep=2pt, circle] (6) at (0,0) [draw] {};
\node[inner sep=2pt, circle] (0) at (0,2) [draw] {};
\node[inner sep=2pt, circle] (1) at (0,4) [draw] {};
\node[inner sep=2pt, circle] (3) at (-2,2) [draw] {};
\node[inner sep=2pt, circle] (2) at (2,2) [draw] {};

\path[style=semithick] (1) edge node[anchor=west]{\small{$1$}}(0);
\path[style=semithick] (0) edge node[anchor=east]{\small{$7$}}(6);
\path[style=semithick] (3) edge node[anchor=north]{\small{$3$}}(0);
\path[style=semithick] (0) edge node[anchor= north]{\small{$2$}}(2);

\node[anchor = south east]  at (0) {\small{$0$}};
\node[anchor = south]  at (1) {\small{$1$}};
\node[anchor = east]  at (3) {\small{$3$}};
\node[anchor = west]  at (2) {\small{$2$}};
\node[anchor = north]  at (4) {\small{$4$}};

\end{tikzpicture}
$$

\end{example}

\noindent Conversely, persistence equivalence does not imply graph equivalence.

\begin{example}

$$
\begin{tikzpicture}[scale=.75]

\node[inner sep=2pt, circle] (4) at (0,0) [draw] {};
\node[inner sep=2pt, circle] (0) at (0,2) [draw] {};
\node[inner sep=2pt, circle] (1) at (0,4) [draw] {};
\node[inner sep=2pt, circle] (5) at (-2,2) [draw] {};
\node[inner sep=2pt, circle] (2) at (2,2) [draw] {};

\path[style=semithick] (1) edge node[anchor=west]{\small{$1$}}(0);
\path[style=semithick] (0) edge node[anchor=east]{\small{$4$}}(4);
\path[style=semithick] (5) edge node[anchor=north]{\small{$5$}}(0);
\path[style=semithick] (0) edge node[anchor= north]{\small{$3$}}(2);

\node[anchor = south east]  at (0) {\small{$0$}};
\node[anchor = south]  at (1) {\small{$1$}};
\node[anchor = east]  at (5) {\small{$5$}};
\node[anchor = west]  at (2) {\small{$2$}};
\node[anchor = north]  at (4) {\small{$4$}};

\end{tikzpicture}
$$

$$
\begin{tikzpicture}[scale=.75]

\node[inner sep=2pt, circle] (5) at (0,0) [draw] {};
\node[inner sep=2pt, circle] (0) at (0,2) [draw] {};
\node[inner sep=2pt, circle] (1) at (0,4) [draw] {};
\node[inner sep=2pt, circle] (2) at (-2,2) [draw] {};
\node[inner sep=2pt, circle] (4) at (2,2) [draw] {};

\path[style=semithick] (1) edge node[anchor=west]{\small{$1$}}(0);
\path[style=semithick] (0) edge node[anchor=east]{\small{$5$}}(5);
\path[style=semithick] (2) edge node[anchor=north]{\small{$3$}}(0);
\path[style=semithick] (0) edge node[anchor= north]{\small{$4$}}(4);

\node[anchor = south east]  at (0) {\small{$0$}};
\node[anchor = south]  at (1) {\small{$1$}};
\node[anchor = north]  at (5) {\small{$5$}};
\node[anchor = east]  at (2) {\small{$2$}};
\node[anchor = west]  at (4) {\small{$4$}};

\end{tikzpicture}
$$

\end{example}

\section{Counting persistence equivalence classes}\label{Counting persistence equivalence classes}

\subsection{An upper bound}

We first prove an upper bound for the number of persistence equivalence classes for any connected graph.  In general, this upper bound is not sharp, as we illustrate in Example \ref{ex not sharp}.  However, we will see in Theorem \ref{thm counting on trees} that this upper bound is sharp on a certain class of graphs, namely, trees.  First, a lemma.

\begin{lemma}\label{lem crit even} Let $f\colon G \to [0,n]$ be a discrete Morse function with $m$ critical values on a connected graph $G$.  Then $m=1+b_1(G)+2k$ for some $k\in \ZZ$.
\end{lemma}

\begin{proof}Let $m=m_0+m_1$, and suppose that $m=2j+1$, as the case when $m$ is even is similar.  By the Theorem \ref{thm weak dmi} (i), $m_0\geq b_0$ and $m_1\geq b_1$ so that $m=b_0+b_1+h$.  By part (ii) of that same theorem, $b_0-b_1=m_0-m_1$.  Adding this to $2j+1=m_0+m_1$, we obtain $2j+2-b_1=2m_0$ so that $b_1=2\ell$ is even.  But if $G$ is connected, $b_0=1$ and $2j+1=m=1+ 2\ell +h$.  Hence $h$ is even.
\end{proof}

\begin{proposition}\label{prop general upper} Let $G$ be a connected graph on $n$ vertices, and let $m:=1+b_1+2k$ where $m \leq n.$ Then there are at most
$$\frac{{n-1 \choose b_1}{n-1-b_1 \choose 2}{n-1-b_1-2 \choose 2}{n-1-b_1-4 \choose 2}\ldots {n-1-b_1-2k+2 \choose 2}}{k!}$$
persistence equivalence classes of discrete Morse functions with $m$ critical values on $G$.
\end{proposition}

\begin{proof} We compute the upper bound by counting all possible barcodes that could be obtained.  The minimum value (in this case $0$) of every discrete Morse function corresponds to a critical vertex, which in turn induces a birth at $t=0$, leaving $n-1$ other times for births and deaths. Since each independent cycle is born and never dies, we have ${n-1 \choose b_1}$ choices of times to birth cycles.  The other $2k$ critical values correspond to persistence pairs.  For the first persistence pair, we choose a birth and death time from the remaining $n-1-b_1$ options, which yields ${n-1-b_1 \choose 2}$ options. There are then ${n-1-b_1-2 \choose 2}$ options for the next persistence pair. Continuing in this manner, we obtain
$$
{n-1 \choose b_1}{n-1-b_1 \choose 2}{n-1-b_1-2 \choose 2}{n-1-b_1-4 \choose 2}\ldots {n-1-b_1-2k+2 \choose 2}.
$$
However, the order in which we choose birth death pairs does not matter, so we must divide by the number of permutations on the number of persistence pairs chosen i.e., divide by $k!$.  Thus the result.
\end{proof}

\noindent Unfortunately this result is not sharp for all graphs.

\begin{example}\label{ex not sharp} Let $C_6$ be a cycle of length $6$, and consider the barcode

$$
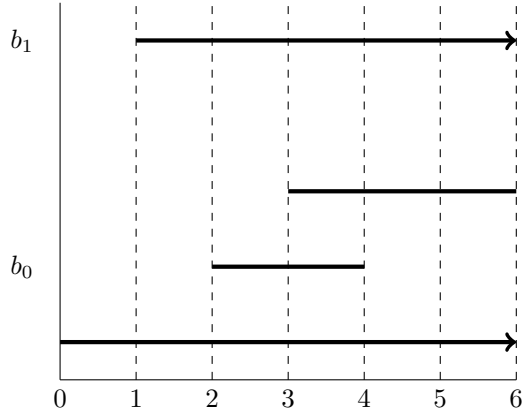
\begin{tikzpicture}[scale=.5][thick]
    \draw (0,0) -- (12,0);
    \draw (0,0) -- (0,10);
    \draw[thin, dashed] (2,0) -- (2,10);
    \draw[thin, dashed] (4,0) -- (4,10);
    \draw[thin, dashed] (6,0) -- (6,10);
    \draw[thin, dashed] (8,0) -- (8,10);
    \draw[thin, dashed] (10,0) -- (10,10);
    \draw[thin, dashed] (12,0) -- (12,10);

    \node at (0,-.5) {$0$};
    \node at (2,-.5) {$1$};
    \node at (4,-.5) {$2$};
    \node at (6,-.5) {$3$};
    \node at (8,-.5) {$4$};
    \node at (10,-.5) {$5$};
    \node at (12,-.5) {$6$};

    \node[black] at (-1,3) {$b_0$};

    \node[black] at (-1,9) {$b_1$};

    \draw[black, ultra thick, ->] (0,1) -- (12,1);
    \draw[black, ultra thick, ->] (2,9) -- (12,9);
    \draw[black, ultra thick] (4,3) -- (8,3);
    \draw[black, ultra thick] (6,5) -- (12,5);

\end{tikzpicture}
$$
This barcode is certainly counted as a possibility in Proposition \ref{prop general upper}.  However, it cannot be obtained on $C_6$ since in order to have the cycle born at time $1$, the entire graph must be built and hence, it is impossible to have any more births and deaths after the cycle is born. Thus, there is only one barcode with the cycle born at $t=1$ on a cycle of any length.
\end{example}

\noindent In the special case of trees, we obtain

\begin{corollary}\label{cor tree count} Let $T$ be a tree on $n$ simplices, $m=2k+1$ an integer, $1\leq m \leq n$. Then there are at most
$$
\frac{(n-1)(n-2)(n-3)\ldots(n-2k)}{2^{k}k!}
$$
persistence equivalence classes of discrete Morse functions on $T$ with $m$ critical values.
\end{corollary}

\begin{proof} By Theorem \ref{thm char of trees}, $b_1(T)=0$ so that Proposition \ref{prop general upper} becomes
$$\frac{{n-1 \choose 2}{n-3 \choose 2}...{n-2k+1 \choose 2}}{k!}.$$

\noindent Observe that ${n-\ell \choose 2} = \frac{(n-\ell)(n-\ell-1)}{2}$.  Hence, replace each factor in the product and simplify to obtain

$$
\frac{(n-1)(n-2)(n-3)\ldots (n-m+2)(n-2k)}{2^{k}k!}.
$$
\end{proof}

\noindent As we will see in Theorem \ref{thm counting on trees}, this upper bound is attained. The next section is devoted to constructing such discrete Morse functions.

\subsection{Counting on trees}

We begin by fixing some notation. Let $T$ be a tree on $n$ simplices.  We say that a persistence diagram $D=\{(c_i,d_i)\}$ is \textbf{consistent with $T$} if $(0,\infty)\in D$ and for all $c_i,d_j$ appearing in ordered pairs in $D-\{(0,\infty)\}$, we have
\begin{enumerate}
    \item $c_i,d_j\in \ZZ$
    \item $1\leq c_i,d_j\leq n$
    \item $c_i<d_i$, and $c_i<c_j$ for all $i<j$
    \item All $c_i,d_j$ are distinct.
\end{enumerate}

\noindent Given a tree $T$ and a persistence diagram $D$ consistent with $T$, we will show that there exists a discrete Morse function on $T$ such that $D_f=D$.

\begin{definition}\label{def S[i]} Let $f \colon T \to [0,n]$ be a discrete Morse function. For a fixed level subcomplex $T_{c_j}$, let $S[v]$ denote the tree of $T_{c_j}$ whose minimum critical vertex is $v$.
\end{definition}

\noindent The following Lemma allows us to compute the persistence diagram using the critical values of the discrete Morse function.

\begin{lemma}\label{lem compute persistence}
Let $f\colon T \to [0,n]$ be a discrete Morse function on a tree. Then $v$ is born at $c_i$  if and only if $f(v)=c_i$ is critical.\\

\noindent Furthermore, $v$ dies at $c_{j+1}$ if and only if there exists a critical edge $e$ with $f(e)=c_{j+1}$ where $e$ joins trees $S[v],S[u]$ in $T_{c_{j+1}}$ with $f(u)<f(v)$.
\end{lemma}

\noindent In other words, when two trees are joined by a critical edge, the vertex that dies is the one that belongs to the tree with larger minimum value.

\begin{proof}
Suppose vertex $v$ is born at $T_{c_i}$.  Then $[v]\not \in \mathrm{im}(f_0^{c_{i-1},c_i})$.  We claim that $v$ is critical. If not, then $v$ is regular and hence part of a free pair, $v<e$.  Write $e=uv$, and suppose that $u\in T_{c_{i-1}}.$ Then $f_0^{c_{i-1},c_i}([u])=[v]$, a contradiction.  If $u \not \in T_{c_{i-1}}$, then $u$ is part of a free pair, and the result follows inductively.  For the converse, we clearly have that $[v]\not \in \mathrm{im}(f_0^{c_{i-1},c_i})$ since $v\not \in T_{c_{i-1}}$ and there does not exist a path from $v$ to any other vertex in $T_{c_i}$.\\

\noindent Suppose that $v$ dies at $T_{c_{j+1}}$ and write $f(e)=c_{j+1}$. For the same reason as above, $e$ must be critical.  Now $e$ connects two trees $S(v)$ and $S(u)$ which were disconnected in $T_{c_{j}}$. By definition of $v$ dying at $c_{j+1}$, $f_0^{c_i,c_{j+1}}(v)$ is the  class $[u]$ that existed at  $T_{c_{i-1}}$.  Since $f(\sigma)\leq c_{i-1}$ for every $\sigma \in T_{c_{i-1}}$, $f(u)\leq c_{i-1}<f(v)$. For the converse, since there is now a path between $u$ and $v$, $[u]=[v]$ in $H_0^{i,j+1}$. Furthermore, because $f(u)<f(v)$, $f^{i,j}_0[v]\not \in \mathrm{im}(f^{i-1,j}_0)$.  Thus $[v]$ dies at time $j+1$.
\end{proof}

\noindent Our construction will break a tree down into a forest and build a discrete Morse function with a single critical value on each tree of the forest.  The following lemma allows us to accomplish this. For the purposes of this lemma, we allow a discrete Morse function to take on any non-negative value for its minimum.

\begin{lemma}\label{lem main lem}
Given a tree $T$ and a vertex $v \in T$, a single value $f(v)$, and some number $n_0$ satisfying $n_0 > f(v)\geq0$, we can extend $f \colon T \rightarrow [f(v),n_0]$ such that $f$ is a discrete Morse function with unique critical simplex $v$ and $\max(f) < n_0$.
\end{lemma}

\begin{proof}
If $T$ is a tree with a single vertex, it follows that $f(v) < n_0$.  Otherwise, let $T$ be a tree with $\ell$ vertices and let $n_0$ be as above. Define  $\alpha_i$ recursively by $\alpha_{0}:= f(v)$ and  $\alpha_{i}:= \frac{n_0+\alpha_{i-1}}{2}$ for $1\leq i \leq \ell$.  It is then clear that $\alpha_0<\alpha_{1}<\ldots <\alpha_{\ell}<n_0$.\\

\noindent Now we wish to extend $f$ to all of $T$. Define $f$ by the following:
\begin{enumerate}
    \item[i.] for any vertex $u \neq v$, label $f(u):=\alpha_{d(v, u)} +f(v)$
\item[ii.] for any edge $e=uw$, label $f(e) := \max\{f(u), f(w)\}$.
\end{enumerate}

\noindent where $d(u,v)$ is the distance between vertices $u$ and $v$; that is, the minimum number of edges over all paths between $u$ and $v$. Note that since $T$ is a tree with $\ell$ vertices, $d(u,v)\leq \ell$.  We need to verify that $f$ is a discrete Morse function with unique critical simplex $v$ and $\max(f) < n_0$. Clearly $v$ is critical since for any edge $e=uv$, we have $f(e)=\max\{f(u), f(v)\} = \alpha_{d(v, u)} +f(v) > f(v)$.\\

\noindent Next, we show that any vertex $u\neq v$ is regular. Let $v_1, \ldots, v_j$ be the neighbors of $u$.  We claim that $f(u)<f(v_i)$ for all $i$ other than exactly one value.
Since $T$ is a tree, there is a unique path from $v$ to $u$. Moreover, this path must pass through exactly one of the $v_i$. Now the unique path from $v$ to any other neighbor $v_k$ of $u$,  must go through $u$. Otherwise, we obtain a cycle. It follows that $d(v,v_i)<d(v,u)<d(v,v_k)$ for some $i$ with $1\leq i\leq j$ and for all $k$ with $1\leq k \leq j$, $k\neq i$, whence $f(v_i)<f(u)<f(v_k)$. Hence $f(u)<f(v_i)$ for all but exactly one value so that all vertices $u \neq v$ are regular.
\\

\noindent Finally, for any edge $e=uw$ in $T$, $f(e) = \max\{f(u), f(w)\}$. Hence we need to show that $f(u)\neq f(w)$. By contradiction, assume that $f(u)=f(w)$. Then $d(u,v) = d(u,w)$, which implies that the path from $v$ to $u$ and the path from $v$ to $w$ along with edge $e$, is a cycle, a contradiction. We conclude that $f$ is a discrete Morse function with the desired properties.
\end{proof}

\noindent The following lemma is clear.

\begin{lemma}\label{lem tree edge}
Let $T$ be a tree, $E:=\{e_1, e_2, \ldots, e_k\}$ a set of edges of $T$, and  $F:=T-E$ the resulting forest. Let $\widetilde{T}$ be any tree in $F$. Then there exists an edge $e\in E$  with one endpoint in $\widetilde{T}$, and the other endpoint in a different tree of the forest.
\end{lemma}

\noindent We now come to our main result.

\begin{theorem}\label{thm counting on trees}
Let $T$ be a tree, $D:=\{(c_i,d_i)\}$ a persistence diagram consistent with $T$.  Then there exists a discrete Morse function $f\colon T\to [0,n]$ such that $D_f=D$.
\end{theorem}

\begin{proof} Let $T$ be a tree and $D:=\{(c_i,d_i)\}$ a persistence diagram consistent with $T$.  Order and label the values in each persistence pair of the persistence diagram  as $0=a_0 < a_1 <\ldots <a_{m-1}<a_m=\infty$. Remove any $\frac{m-1}{2}$ edges $E$ from $T$ and write $F:=T-E$. We will label $T$ with a discrete Morse function by inducing on $c_i$, the birth times, and continually extending the function $f$ on subgraphs of $T$ until it is defined on all of $T$.\\

\noindent For $c_0=0$, pick any tree $T_0$ in $F$ and any vertex $v_0$ of $T_0$. Define $f(v_0):=0$.  Applying Lemma \ref{lem main lem} on $T_0$ with $n_0=a_1$, we obtain a discrete Morse function on $T_0$ with $v_0$ the unique critical vertex. In the case where $a_1=\infty$, pick any finite $n>0$ to obtain a labeling of the entire tree $T$.\\

\noindent In the case where $c_1=a_1\neq \infty$, let $(c_1,d_1)\in D$. Apply Lemma \ref{lem tree edge} on $T$ and $E$ with $\widetilde{T}=T_0$ to obtain an edge $e_1$ joining $T_0$ and an unlabeled tree $T_1$. Pick any $v_1\in T_1$ and again applying Lemma \ref{lem main lem} on $T_1$ with $f(v_1):=c_1$ and $n_0=a_2$. Then we extend $f$ to obtain a discrete Morse function on $T_0\cup T_1$ with $v_1$ also critical. Furthermore,  since $\max{f}<a_2$ and $a_2\leq d_1$, if we label $f(e_1):=d_1$,  it follows that $e_1$ is a critical edge.\\

\noindent In general, let $(c_i,d_i)\in D$ with $i>1$ and suppose that $c_i=a_j$. We again apply Lemma \ref{lem tree edge} to $T$ on $\widetilde{E}$, the subset of $E$ consisting of currently unlabeled edges, and let $\widetilde{F}:=T-\widetilde{E}$ be the resulting forest. Then there is a unique tree $\widetilde{T}$ of $\widetilde{F}$ which is labeled by $f$. The Lemma then guarantees that there is an edge $e_i$ with one endpoint in $\widetilde{T}$ and the the other in a different (unlabeled) tree $T_i$.  Choosing a vertex $v_i\in T_i$, label $f(v_i):=c_i$, and apply Lemma \ref{lem main lem} with $n:=a_{j+1}$ to obtain a discrete Morse function on $\widetilde{T}\cup T_i$.  Furthermore,  since $\max\{f\}<a_j$ and $a_j\leq d_i$, if we label $f(e_i):=d_i$,  it follows that $e_i$ is a critical edge. In this way, we obtain a discrete Morse function on all of $T$. Using the language of Definition \ref{def S[i]}, note that  $T_i=S[v_i]$ by the above construction.\\

\noindent Finally, we need to show that $D_f=D$.  For this, we apply Lemma \ref{lem compute persistence}.  The critical vertices are exactly those labeled $c_i$ so that the birth times are correct.  Consider level subcomplex $T_{d_i}$. By construction, $d_i=f(e_i)$ with edge $e_i$ joining trees $S[v_i]$ and $S[v_j]$ and $f(v_j)<f(v_i)$.  Hence $v_i$ dies at $d_i$ i.e. $(c_i,d_i)$ is a persistence pair of $D_f$.
\end{proof}

\subsection{An example}\label{An example}

\noindent We illustrate the construction given in Theorem \ref{thm counting on trees}.  Let $T$ be the tree

$$
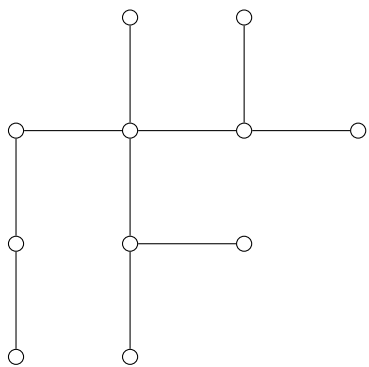
\begin{tikzpicture}[scale=.75]

\node[inner sep=2pt, circle] (5) at (0,0) [draw] {};
\node[inner sep=2pt, circle] (45) at (0,2) [draw] {};
\node[inner sep=2pt, circle] (3) at (0,4) [draw] {};
\node[inner sep=2pt, circle] (4) at (0,6) [draw] {};
\node[inner sep=2pt, circle] (0) at (-2,0) [draw] {};
\node[inner sep=2pt, circle] (1) at (-2,2) [draw] {};
\node[inner sep=2pt, circle] (2) at (-2,4) [draw] {};
\node[inner sep=2pt, circle] (9) at (2,2) [draw] {};
\node[inner sep=2pt, circle] (14) at (2,4) [draw] {};
\node[inner sep=2pt, circle] (145) at (2,6) [draw] {};
\node[inner sep=2pt, circle] (15) at (4,4) [draw] {};

\draw[-]  (0)--(1) node[midway] {};
\draw[-]  (2)--(1) node[midway] {};
\draw[-]  (2)--(3) node[midway] {};
\draw[-]  (3)--(4) node[midway] {};
\draw[-]  (3)--(45) node[midway] {};
\draw[-]  (3)--(14) node[midway] {};
\draw[-]  (45)--(9) node[midway] {};
\draw[-]  (5)--(45) node[midway] {};
\draw[-]  (145)--(14) node[midway] {};
\draw[-]  (15)--(14) node[midway] {};

\end{tikzpicture}
$$

\noindent and barcode given by

$$
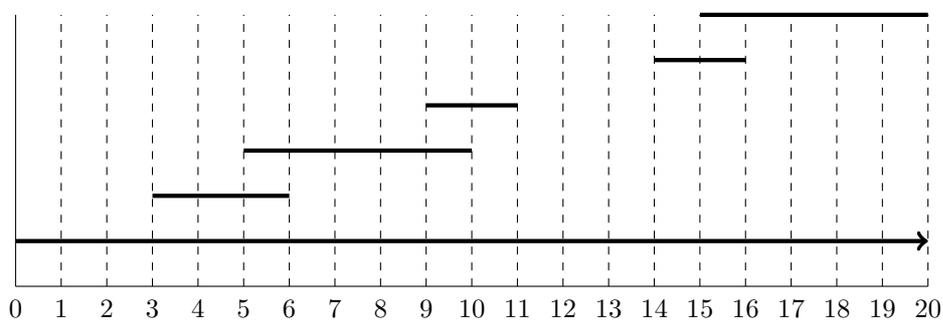
\begin{tikzpicture}[scale=.6][thick]
    \draw (0,0) -- (20,0);
    \draw (0,0) -- (0,6);
    \draw[thin, dashed] (1,0) -- (1,6);
    \draw[thin, dashed] (2,0) -- (2,6);
    \draw[thin, dashed] (3,0) -- (3,6);
    \draw[thin, dashed] (4,0) -- (4,6);
    \draw[thin, dashed] (5,0) -- (5,6);
    \draw[thin, dashed] (6,0) -- (6,6);
    \draw[thin, dashed] (7,0) -- (7,6);
    \draw[thin, dashed] (8,0) -- (8,6);
    \draw[thin, dashed] (9,0) -- (9,6);
    \draw[thin, dashed] (10,0) -- (10,6);
    \draw[thin, dashed] (11,0) -- (11,6);
    \draw[thin, dashed] (12,0) -- (12,6);
    \draw[thin, dashed] (13,0) -- (13,6);
    \draw[thin, dashed] (14,0) -- (14,6);
    \draw[thin, dashed] (15,0) -- (15,6);
    \draw[thin, dashed] (16,0) -- (16,6);
    \draw[thin, dashed] (17,0) -- (17,6);
    \draw[thin, dashed] (18,0) -- (18,6);
    \draw[thin, dashed] (19,0) -- (19,6);
    \draw[thin, dashed] (20,0) -- (20,6);

    \node at (0,-.5) {$0$};
    \node at (1,-.5) {$1$};
    \node at (2,-.5) {$2$};
    \node at (3,-.5) {$3$};
    \node at (4,-.5) {$4$};
    \node at (5,-.5) {$5$};
    \node at (6,-.5) {$6$};
    \node at (7,-.5) {$7$};
    \node at (8,-.5) {$8$};
    \node at (9,-.5) {$9$};
    \node at (10,-.5) {$10$};
    \node at (11,-.5) {$11$};
    \node at (12,-.5) {$12$};
    \node at (13,-.5) {$13$};
    \node at (14,-.5) {$14$};
    \node at (15,-.5) {$15$};
    \node at (16,-.5) {$16$};
    \node at (17,-.5) {$17$};
    \node at (18,-.5) {$18$};
    \node at (19,-.5) {$19$};
    \node at (20,-.5) {$20$};

    \draw[ultra thick, ->] (0,1) -- (20,1);
    \draw[ultra thick] (3,2) -- (6,2);
    \draw[ultra thick] (5,3) -- (10,3);
    \draw[ultra thick] (9,4) -- (11,4);
    \draw[ultra thick] (14,5) -- (16,5);
    \draw[ultra thick] (15,6) -- (20,6);

\end{tikzpicture}
$$

\noindent We will use the method of Theorem \ref{thm counting on trees} to construct a discrete Morse function $f$ of $T$ that induces the above barcode. Following the proof, we first order the critical values from the barcode
$$
0<3<5<6<9<10<11<14<15<16<20<\infty.
$$
\noindent Next we remove $\frac{m-1}{2}=\frac{11-1}{2}=5$ edges from $T$ to obtain the forest $F$ below.

$$
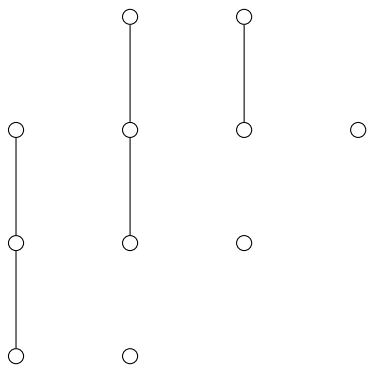
\begin{tikzpicture}[scale=.75]

\node[inner sep=2pt, circle] (5) at (0,0) [draw] {};
\node[inner sep=2pt, circle] (45) at (0,2) [draw] {};
\node[inner sep=2pt, circle] (3) at (0,4) [draw] {};
\node[inner sep=2pt, circle] (4) at (0,6) [draw] {};
\node[inner sep=2pt, circle] (0) at (-2,0) [draw] {};
\node[inner sep=2pt, circle] (1) at (-2,2) [draw] {};
\node[inner sep=2pt, circle] (2) at (-2,4) [draw] {};
\node[inner sep=2pt, circle] (9) at (2,2) [draw] {};
\node[inner sep=2pt, circle] (14) at (2,4) [draw] {};
\node[inner sep=2pt, circle] (145) at (2,6) [draw] {};
\node[inner sep=2pt, circle] (15) at (4,4) [draw] {};

\draw[-]  (0)--(1) node[midway] {};
\draw[-]  (2)--(1) node[midway] {};
\draw[-]  (3)--(4) node[midway] {};
\draw[-]  (3)--(45) node[midway] {};
\draw[-]  (145)--(14) node[midway] {};

\end{tikzpicture}
$$

\noindent Our first step in labeling is to pick a tree in $F$ and label a vertex $0$.  We then apply Lemma \ref{lem main lem} with $n_0=a_1=3$ to obtain

$$
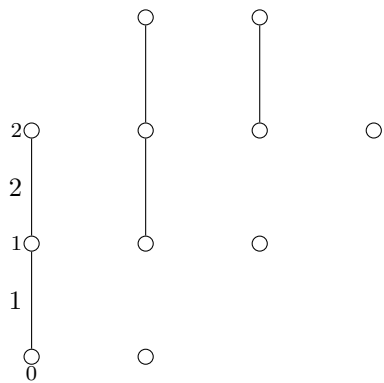
\begin{tikzpicture}[scale=.75]

\node[inner sep=2pt, circle] (5) at (0,0) [draw] {};
\node[inner sep=2pt, circle] (45) at (0,2) [draw] {};
\node[inner sep=2pt, circle] (3) at (0,4) [draw] {};
\node[inner sep=2pt, circle] (4) at (0,6) [draw] {};
\node[inner sep=2pt, circle] (0) at (-2,0) [draw] {};
\node[inner sep=2pt, circle] (1) at (-2,2) [draw] {};
\node[inner sep=2pt, circle] (2) at (-2,4) [draw] {};
\node[inner sep=2pt, circle] (9) at (2,2) [draw] {};
\node[inner sep=2pt, circle] (14) at (2,4) [draw] {};
\node[inner sep=2pt, circle] (145) at (2,6) [draw] {};
\node[inner sep=2pt, circle] (15) at (4,4) [draw] {};

\draw[-]  (0)--(1) node[midway, left] {$1$};
\draw[-]  (2)--(1) node[midway, left] {$2$};
\draw[-]  (3)--(4) node[midway] {};
\draw[-]  (3)--(45) node[midway] {};
\draw[-]  (145)--(14) node[midway] {};

\node[anchor = north]  at (0) {\small{$0$}};
\node[anchor = east]  at (1) {\small{$1$}};
\node[anchor = east]  at (2) {\small{$2$}};

\end{tikzpicture}
$$

\noindent The first persistence pair with both values finite is $(3,6)$. By Lemma \ref{lem tree edge}, there is a removed edge connecting the labeled tree with an unlabeled tree (in this case, there is only one such edge). We pick a vertex in the unlabeled tree and label it $3$.  Applying the same Lemma as above with $n_0=a_2=5$ and furthermore labeling the connecting edge by $6$, we have

$$
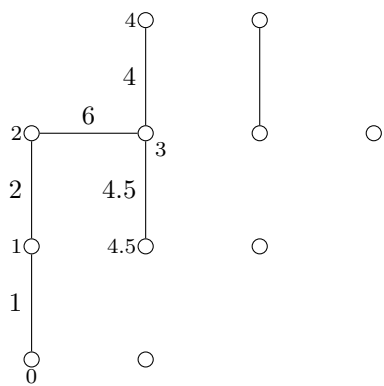
\begin{tikzpicture}[scale=.75]

\node[inner sep=2pt, circle] (5) at (0,0) [draw] {};
\node[inner sep=2pt, circle] (45) at (0,2) [draw] {};
\node[inner sep=2pt, circle] (3) at (0,4) [draw] {};
\node[inner sep=2pt, circle] (4) at (0,6) [draw] {};
\node[inner sep=2pt, circle] (0) at (-2,0) [draw] {};
\node[inner sep=2pt, circle] (1) at (-2,2) [draw] {};
\node[inner sep=2pt, circle] (2) at (-2,4) [draw] {};
\node[inner sep=2pt, circle] (9) at (2,2) [draw] {};
\node[inner sep=2pt, circle] (14) at (2,4) [draw] {};
\node[inner sep=2pt, circle] (145) at (2,6) [draw] {};
\node[inner sep=2pt, circle] (15) at (4,4) [draw] {};

\draw[-]  (0)--(1) node[midway, left] {$1$};
\draw[-]  (2)--(1) node[midway, left] {$2$};
\draw[-]  (2)--(3) node[midway, above] {$6$};
\draw[-]  (3)--(4) node[midway, left] {$4$};
\draw[-]  (3)--(45) node[midway, left] {$4.5$};
\draw[-]  (145)--(14) node[midway] {};

\node[anchor = north]  at (0) {\small{$0$}};
\node[anchor = east]  at (1) {\small{$1$}};
\node[anchor = east]  at (2) {\small{$2$}};
\node[anchor = north west]  at (3) {\small{$3$}};
\node[anchor = east]  at (4) {\small{$4$}};
\node[anchor = east]  at (45) {\small{$4.5$}};

\end{tikzpicture}
$$
\noindent We continue in this manner.  The next persistence pair is $(5,10)$.  Choosing the lowest tree, we label the sole vertex $5$ and the edge $10$.

$$
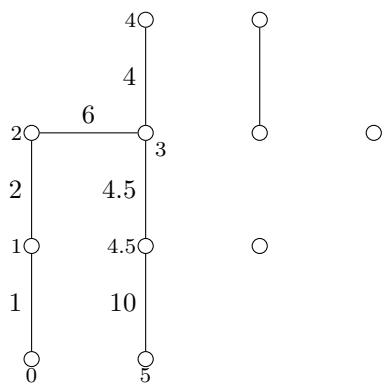
\begin{tikzpicture}[scale=.75]

\node[inner sep=2pt, circle] (5) at (0,0) [draw] {};
\node[inner sep=2pt, circle] (45) at (0,2) [draw] {};
\node[inner sep=2pt, circle] (3) at (0,4) [draw] {};
\node[inner sep=2pt, circle] (4) at (0,6) [draw] {};
\node[inner sep=2pt, circle] (0) at (-2,0) [draw] {};
\node[inner sep=2pt, circle] (1) at (-2,2) [draw] {};
\node[inner sep=2pt, circle] (2) at (-2,4) [draw] {};
\node[inner sep=2pt, circle] (9) at (2,2) [draw] {};
\node[inner sep=2pt, circle] (14) at (2,4) [draw] {};
\node[inner sep=2pt, circle] (145) at (2,6) [draw] {};
\node[inner sep=2pt, circle] (15) at (4,4) [draw] {};

\draw[-]  (0)--(1) node[midway, left] {$1$};
\draw[-]  (2)--(1) node[midway, left] {$2$};
\draw[-]  (2)--(3) node[midway, above] {$6$};
\draw[-]  (3)--(4) node[midway, left] {$4$};
\draw[-]  (3)--(45) node[midway, left] {$4.5$};
\draw[-]  (5)--(45) node[midway, left] {$10$};
\draw[-]  (145)--(14) node[midway] {};

\node[anchor = north]  at (0) {\small{$0$}};
\node[anchor = east]  at (1) {\small{$1$}};
\node[anchor = east]  at (2) {\small{$2$}};
\node[anchor = east]  at (4) {\small{$4$}};
\node[anchor = east]  at (45) {\small{$4.5$}};
\node[anchor = north]  at (5) {\small{$5$}};
\node[anchor = north west]  at (3) {\small{$3$}};

\end{tikzpicture}
$$

\noindent Three more iterations of this step yields the following discrete Morse function.

$$
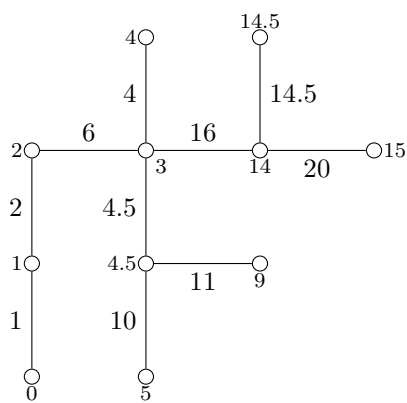
\begin{tikzpicture}[scale=.75]

\node[inner sep=2pt, circle] (5) at (0,0) [draw] {};
\node[inner sep=2pt, circle] (45) at (0,2) [draw] {};
\node[inner sep=2pt, circle] (3) at (0,4) [draw] {};
\node[inner sep=2pt, circle] (4) at (0,6) [draw] {};
\node[inner sep=2pt, circle] (0) at (-2,0) [draw] {};
\node[inner sep=2pt, circle] (1) at (-2,2) [draw] {};
\node[inner sep=2pt, circle] (2) at (-2,4) [draw] {};
\node[inner sep=2pt, circle] (9) at (2,2) [draw] {};
\node[inner sep=2pt, circle] (14) at (2,4) [draw] {};
\node[inner sep=2pt, circle] (145) at (2,6) [draw] {};
\node[inner sep=2pt, circle] (15) at (4,4) [draw] {};

\draw[-]  (0)--(1) node[midway, left] {$1$};
\draw[-]  (2)--(1) node[midway, left] {$2$};
\draw[-]  (2)--(3) node[midway, above] {$6$};
\draw[-]  (3)--(4) node[midway, left] {$4$};
\draw[-]  (3)--(45) node[midway, left] {$4.5$};
\draw[-]  (3)--(14) node[midway, above] {$16$};
\draw[-]  (45)--(9) node[midway, below] {$11$};
\draw[-]  (5)--(45) node[midway, left] {$10$};
\draw[-]  (145)--(14) node[midway, right] {$14.5$};
\draw[-]  (15)--(14) node[midway, below] {$20$};

\node[anchor = north]  at (0) {\small{$0$}};
\node[anchor = east]  at (1) {\small{$1$}};
\node[anchor = east]  at (2) {\small{$2$}};
\node[anchor = east]  at (4) {\small{$4$}};
\node[anchor = east]  at (45) {\small{$4.5$}};
\node[anchor = north]  at (5) {\small{$5$}};
\node[anchor = north]  at (9) {\small{$9$}};
\node[anchor = north]  at (14) {\small{$14$}};
\node[anchor = south]  at (145) {\small{$14.5$}};
\node[anchor = west]  at (15) {\small{$15$}};
\node[anchor = north west]  at (3) {\small{$3$}};

\end{tikzpicture}
$$

\noindent While Theorem \ref{thm counting on trees} guarantees this discrete Morse function induces the desired barcode, it is also easy to check by hand that this discrete Morse function induces the desired barcode.


\end{document}